\documentclass{article}
\usepackage{amsmath}
\usepackage{amsthm}
\usepackage{amssymb}

\theoremstyle{definition}
\newtheorem{defi}{Definition}

\theoremstyle{plain}
\newtheorem{thm}[defi]{Theorem}
\newtheorem{cor}[defi]{Corollary}

\newtheorem{prop}[defi]{Proposition}
\newtheorem{claim}[defi]{Claim}
\newtheorem{ques}[defi]{Question}
\newtheorem{obs}[defi]{Observation}

\newcounter{enuroman}
\renewcommand{\theenuroman}{\roman{enuroman}}
\newenvironment{romanenumerate}{\begin{list}{\rm (\theenuroman)}{\usecounter{enuroman}
    \setlength{\labelwidth}{1cm}}}
   {\end{list}}
\newenvironment{romanenumerate'}{\begin{list}{\rm (\theenuroman')}{\usecounter{enuroman}
    \setlength{\labelwidth}{1cm}}}
   {\end{list}}

\newcommand{\forces}{\Vdash}
\newcommand{\re}{{\upharpoonright}}

\newcommand{\A}{{\cal A}}

\newcommand{\F}{{\cal F}}

\newcommand{\I}{{\cal I}}

\newcommand{\M}{{\cal M}}

\newcommand{\NWD}{{\cal{NWD}}}

\newcommand{\CC}{{\mathbb C}}

\newcommand{\bb}{{\mathfrak b}}

\newcommand{\cc}{{\mathfrak c}}
\newcommand{\dd}{{\mathfrak d}}

\newcommand{\add}{{\mathsf{add}}}
\newcommand{\cov}{{\mathsf{cov}}}
\newcommand{\non}{{\mathsf{non}}}
\newcommand{\cof}{{\mathsf{cof}}}

\newcommand{\cf}{{\mathrm{cf}}}

\newcommand{\lh}{{\mathrm{lh}}}

\newcommand{\sub}{\subseteq}
\newcommand{\sem}{\setminus}
\newcommand{\twoom}{2^\omega}

\newcommand{\omom}{\omega^\omega}

\newcommand{\ha}{\,{}\hat{}\,}
\newcommand{\la}{\langle}
\newcommand{\ra}{\rangle}

\newcommand{\Tau}{\mathrm{T}}

\title{The higher Cicho\'{n} diagram in the degenerate case}

\author{J\"org Brendle\thanks{Partially supported by Grants-in-Aid for Scientific Research
   (C) 15K04977 and 18K03398, Japan Society for the Promotion of Science.\newline
    }  \\
   Graduate School of System Informatics \\
   Kobe University \\
   Rokko-dai 1-1, Nada-ku \\
   Kobe 657-8501, Japan \\
   email: {\sf brendle@kobe-u.ac.jp}}

\begin{document}
\maketitle

\begin{abstract}
\noindent 
For a regular uncountable cardinal $\kappa$, we discuss the order relationship between the 
unbounding and dominating numbers $\bb_\kappa$ and $\dd_\kappa$ on $\kappa$ and
cardinal invariants of the higher meager ideal $\M_\kappa$. In particular, we obtain an almost complete
characterization of $\add (\M_\kappa)$ and $\cof (\M_\kappa)$ in terms of $\cov (\M_\kappa)$ and
$\non (\M_\kappa)$ and unbounding and dominating numbers, and we provide models showing
that there are no restrictions on the value of $\non (\M_\kappa)$ in the degenerate case $2^{< \kappa} > \kappa$
except $2^{<\kappa} \leq \non (\M_\kappa) \leq 2^\kappa$. The corresponding question for
$\cof (\M_\kappa)$ remains open. Our results answer questions of joint work of the author with
Brooke-Taylor, Friedman, and Montoya~\cite[Questions 29 and 32]{BBFM18}.
\end{abstract}


\section{Introduction}

Cardinal invariants of the continuum, describing the combinatorial properties of the real numbers ($\twoom$ or $\omom$) and taking
values between $\omega_1$ and the continuum $\cc$, have been studied intensively for several decades, and a rich theory with ZFC-results
and independence proofs about the order relationship between various cardinal invariants has evolved (see~\cite{BJ95} and~\cite{Bl10}).
More recently, {\em higher cardinal invariants}, that is, cardinal invariants of the higher Cantor space $2^\kappa$ or the higher Baire
space $\kappa^\kappa$, where $\kappa$ is an uncountable regular cardinal, have started to be investigated and our work is
a contribution to this ongoing research. 

Our focus lies on cardinal invariants in the {\em higher Cicho\'n diagram} (see~\cite{BBFM18} and~\cite{BGSta}). The original
Cicho\'n diagram~\cite{BJ95} describes the relationship between cardinal invariants related to measure and category as well 
as the unbounding and dominating numbers $\bb$ and $\dd$. The latter can be easily redefined in the context of regular
uncountable $\kappa$, by
\begin{itemize}
\item $\bb_\kappa = \min \{ |F| : F \sub \kappa^\kappa$ and $\forall g \in \kappa^\kappa \; \exists f \in F \; ( f \not\leq^* g) \}$, \\ the
   {\em $\kappa$-unbounding number}, and
\item $\dd_\kappa = \min \{ |F| : F \sub \kappa^\kappa$ and $\forall g \in \kappa^\kappa \; \exists f \in F \; ( g \leq^* f) \}$, \\ the
   {\em $\kappa$-dominating number},
\end{itemize}
where $f \leq^* g$ if there is $\alpha < \kappa$ such that $f (\beta ) \leq g(\beta)$ for all $\beta \geq \alpha$. Clearly $\bb_\kappa \leq \dd_\kappa$.
Also, there is a natural analog of the meager ideal on the higher Cantor space $2^\kappa$: give $2$ the discrete topology and equip
$2^\kappa$ with the {\em $\kappa$-box topology}. That is, basic clopen sets are of the form
\[ [\sigma] = \{ f \in 2^\kappa : \sigma \sub f \} \]
where $\sigma \in 2^{< \kappa}$, and open sets are arbitrary unions of such basic clopen sets. Thus a set $A \sub 2^{< \kappa}$ 
is nowhere dense in this topology if for all $\sigma \in 2^{< \kappa}$ there is $\tau \supseteq \sigma$ such that $[\tau] \cap A 
= \emptyset$. This implies that the nowhere dense ideal on $2^\kappa$, denoted by $\NWD_\kappa$, is $< \kappa$-closed
(i.e. closed under unions of size $< \kappa$). Say that $A \sub 2^\kappa$ is {\em $\kappa$-meager} if it is a union of at most
$\kappa$ many nowhere dense sets, and let $\M_\kappa$ denote the ($\kappa$-closed) ideal of $\kappa$-meager sets. 
It is much less clear how the null ideal should be generalized to regular uncountable $\kappa$; a very interesting candidate has been 
proposed (for weakly compact $\kappa$) by Baumhauer, Goldstern, and Shelah in~\cite{BGSta}. We shall not pursue this here.

Let $\I$ be a non-trivial ideal on a set $X$, that is, all the singletons $\{ x \}$, $x \in X$, belong to $\I$ and $X \notin \I$. Define
\begin{itemize}
\item $\add (\I) = \min \{ | \F| : \F \sub \I$ and $\bigcup \F \notin \I \}$, the {\em additivity of $\I$},
\item $\cov (\I) = \min \{ | \F| : \F \sub \I$ and $\bigcup \F = X\}$, the {\em covering number of $\I$},
\item $\non (\I) = \min \{ | Y | : Y \sub X $ and $Y \notin \I \}$, the {\em uniformity of $\I$}, and
\item $\cof (\I) = \min \{ | \F| : \F \sub \I$ and $\forall A \in \I \; \exists B \in \F \; (A \sub B)  \}$, \\ the {\em cofinality of $\I$}.
\end{itemize}
Easily $\add (\I) \leq \cov (\I) , \non (\I) \leq \cof (\I)$. In our earlier work~\cite{BBFM18} we observed that
$\bb_\kappa \leq \non (\M_\kappa)$ and $\cov (\M_\kappa) \leq \dd_\kappa$~\cite[Observation 17]{BBFM18}, and proved:
\begin{romanenumerate}
\item $\add (\M_\kappa) \leq \bb_\kappa$ and $\dd_\kappa \leq \cof (\M_\kappa)$ for strongly inaccessible $\kappa$~\cite[Corollary 28]{BBFM18},
\item $\add (\M_\kappa) \geq \min \{ \bb_\kappa , \cov (\M_\kappa) \}$~\cite[Corollary 31]{BBFM18}, and
\item $\cof (\M_\kappa) \leq \max \{ \dd_\kappa , \non (\M_\kappa) \}$ in case $2^{< \kappa} = \kappa$~\cite[Corollary 31]{BBFM18}.
\end{romanenumerate}
In particular, $\add (\M_\kappa) = \min \{ \bb_\kappa, \cov (\M_\kappa) \}$ and $\cof (\M_\kappa) = \max \{ \dd_\kappa, \non (\M_\kappa) \}$
for strongly inaccessible $\kappa$, and the cardinals can be displayed in the following diagram.
\begin{figure}[ht]
\begin{center}
\setlength{\unitlength}{0.2000mm}
\begin{picture}(800.0000,180.0000)(80,10)
\thinlines
\put(355,180){\vector(1,0){50}}
\put(355,100){\vector(1,0){50}}
\put(355,20){\vector(1,0){50}}
\put(440,30){\vector(0,1){60}}
\put(440,110){\vector(0,1){60}}
\put(320,110){\vector(0,1){60}}
\put(320,30){\vector(0,1){60}}
\put(410,170){\makebox(60,20){$\cof(\mathcal{M}_\kappa)$}}
\put(410,90){\makebox(60,20){$\mathfrak{d}_\kappa$}}
\put(410,10){\makebox(60,20){$\cov(\mathcal{M}_\kappa)$}}
\put(290,10){\makebox(60,20){$\add(\mathcal{M}_\kappa)$}}
\put(290,90){\makebox(60,20){$\mathfrak{b}_\kappa$}}
\put(290,170){\makebox(60,20){$\non(\mathcal{M}_\kappa)$}}
\end{picture}
\end{center}
\caption{middle part of Cicho\'n's diagram for regular uncountable $\kappa$}
\label{ODiag}
\end{figure}
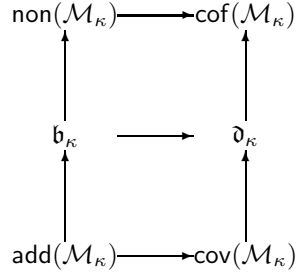
We asked whether the assumptions were necessary in (i) and (iii) above~\cite[Questions 29 and 32]{BBFM18}, and these 
questions are the starting point of the present work.

Note that in the {\em degenerate case} $2^{< \kappa} > \kappa$, some of the cardinal invariants become trivial. Namely Landver~\cite[1.3]{La92}
(see also~\cite[Observation 23(ii)]{BBFM18}) observed that $\add (\M_\kappa) = \cov (\M_\kappa) = \kappa^+$ and this accounts for (ii)
above in case $2^{< \kappa} > \kappa$. Blass, Hyttinen, and Zhang~\cite[4.15]{BHZta} (see also~\cite[Observation 23(iii)]{BBFM18}) noticed
that $2^{< \kappa} \leq \non (\M_\kappa)$. Finally, we proved~\cite[Proposition 2 (c)]{Br17} that $2^{< \kappa} < \cof (\M_\kappa)$ and this
implies that $\cof (\M_\kappa) > \max \{ \dd_\kappa , \non (\M_\kappa) \}$ in any model of $2^{< \kappa} = 2^\kappa$; in particular, (iii) may
fail in the degenerate case. We will obtain better lower bounds for $\cof (\M_\kappa)$ (see in particular Corollary~\ref{firstcor}). The results about
$\add (\M_\kappa)$ and $\cov (\M_\kappa)$ in the degenerate case suggest the problem of whether one can say more about $\non (\M_\kappa)$ and 
$\cof (\M_\kappa)$ in this situation. We shall see this is not the case for $\non (\M_\kappa)$ (see in particular Theorems~\ref{1stcon} 
and~\ref{2ndcon}), while for $\cof (\M_\kappa)$ some interesting questions remain open.

This paper is organized as follows. In Section~\ref{ZFC}, we obtain several ZFC-results about higher cardinal invariants which strengthen
those of~\cite{BBFM18}; in particular we will compute $\add (\M_\kappa)$ and $\cof (\M_\kappa)$ in terms of the other cardinals in most cases
(Corollary~\ref{secondcor}). In Section~\ref{models}, we present independence results for the values of $\non (\M_\kappa)$ and
$\cof (\M_\kappa)$ in the degenerate context. Section~\ref{dominating} investigates dominating numbers naturally arising in the
discussion of $\cof (\M_\kappa)$.

\bigskip

{\sc Preliminaries.} Let $\kappa$ be regular and $\lambda \geq \kappa$. For $f,g \in \kappa^\lambda$ say that $g$ {\em eventually dominates} $f$ ($f \leq^* g$ 
in symbols) if the set $\{ \gamma < \lambda :  f(\gamma) > g(\gamma) \}$ has size less than $\kappa$. Let $\bb^\lambda_\kappa$ and
$\dd^\lambda_\kappa$ be the unbounding and dominating numbers of $(\kappa^\lambda, \leq^*)$, respectively\footnote{As
we will see in Proposition~\ref{dom-prop} in Section~\ref{dominating}, it does not really matter whether we use domination everywhere,
modulo $< \kappa$ or modulo $< \lambda$ for our results; however, for the proofs in Section~\ref{ZFC} the present definition is
most convenient.}. That is,
\begin{itemize}
\item $\bb^\lambda_\kappa = \min \{ |F| : F \sub \kappa^\lambda$ and $\forall g \in \kappa^\lambda \; \exists f \in F \; ( f \not\leq^* g) \}$
\item $\dd^\lambda_\kappa = \min \{ |F| : F \sub \kappa^\lambda$ and $\forall g \in \kappa^\lambda \; \exists f \in F \; ( g \leq^* f) \}$
\end{itemize}
For $\kappa = \lambda$ we have $\bb_\kappa^\lambda = \bb_\kappa$ and $\dd_\kappa^\lambda = \dd_\kappa$.
In general $\dd_\kappa^\lambda \geq \dd_\kappa^\mu \geq \dd_\kappa$ where $\kappa\leq\mu\leq \lambda$.
If $\lambda > \kappa$ then $\bb^\lambda_\kappa = \kappa$, as witnessed by the constant functions.


\section{ZFC results}
\label{ZFC}

For this whole section, let $\kappa$ be regular uncountable and let $\lambda = |2^{< \kappa}|$.

\begin{thm} \label{mainthm}
There are functions $\Phi_- :  \kappa^\lambda \to \NWD_\kappa$ and $\Phi_+ : \M_\kappa \to 
 [\kappa^\lambda]^{\lambda}$ such that if $A \in \M_\kappa$,
$g \in  \kappa^\lambda$, and $g$ is not eventually bounded by
$\Phi_+ (A)$, then $\Phi_- (g) \not\sub A$.
\end{thm}

\begin{proof}
Let $\Sigma = \{ \tilde\sigma \} \cup \{ \sigma_\gamma : \gamma < \lambda \} \sub 2^{<\kappa}$ be a maximal antichain
in $2^{< \kappa}$. 

Fix $g \in \kappa^\lambda$. We recursively define a nowhere dense tree $T_{g} \sub 2^{<\kappa}$.
More explicitly, we define sets $C^\alpha_{g} \sub 2^{< \kappa}$ for $\alpha < \kappa$ such that 
\begin{itemize}
\item $C^\alpha_{g}$ is an antichain in $2^{<\kappa}$,
\item if $\alpha < \beta$ and $\tau \in C^\beta_{g}$ then there is (necessarily unique) $\sigma \in C^\alpha_{g}$ such that $\sigma \subsetneq \tau$,
\item if $\alpha < \beta$ and $\sigma \in C^\alpha_{g}$ then there is $\tau \in C^\beta_{g}$ such that $\sigma \subsetneq \tau$,
\item for each $\sigma \in C^\alpha_{g}$ there is $\tau \supsetneq \sigma$ incompatible with all members of $C^{\alpha + 1}_{g}$.
\end{itemize}
Then let $T_{g}$ be the downward closure of $\bigcup_{\alpha<\kappa} C^\alpha_{g}$, i.e., $\sigma \in T_{g}$ if
$\sigma \sub \tau$ for some $\tau \in C^\alpha_{g}$ and some $\alpha$.
Clearly $T_{g}$ is a nowhere dense tree, and we let $\Phi_-(g) = [ T_{g} ]$.

Let $C^0_g = \{ \la\ra \}$.

Assume $C^\alpha_{g}$ has been defined and let $\sigma \in C^\alpha_{g}$. Assume $\lh (\sigma) = \zeta$.
Then $\tau \supsetneq \sigma$ belongs to $C^{\alpha + 1}_{g}$ if for some $\gamma < \lambda$,
$\sigma \ha \sigma_\gamma \sub \tau$ and $\lh (\tau) = \zeta + \lh (\sigma_\gamma) + g(\gamma)$.
Note that this implies that $\tau = \sigma \ha \tilde \sigma$ is incompatible with all of $C^{\alpha+1}_g$.

If $\alpha$ is a limit ordinal, put $\sigma$ into $C^{\alpha}_{g}$ if there is a strictly increasing
sequence $(\tau_\beta : \beta < \alpha)$ such that $\sigma = \bigcup_{\beta < \alpha} \tau_\beta$ and each
$\tau_\beta$ belongs to $C^\beta_{g}$. This completes the construction of the $C^\alpha_g$ and of $T_g$.

Next fix $A \in \M_\kappa$, $A = \bigcup_{\alpha < \kappa} A_\alpha$, where the $A_\alpha$ form an increasing
sequence of nowhere dense sets. Define $h = h^A : 2^{< \kappa } \to 2^{< \kappa}$ such that
for all $\alpha < \kappa$ and all $\sigma \in 2^\alpha$, $\sigma \sub h(\sigma)$ and $[h(\sigma)] \cap A_\alpha = \emptyset$.

Fix $\sigma \in 2^{< \kappa}$. Say $\lh (\sigma) = \zeta$. Define $f^A_\sigma \in \kappa^\lambda$ such that
$\lh ( h (\sigma \ha \sigma_\gamma)) = \zeta + \lh (\sigma_\gamma) + f^A_\sigma (\gamma)$ for all $\gamma < \lambda$.
Let $\Phi_+ (A) = \{ f^A_\sigma : \sigma \in 2^{< \kappa} \}$. 

Now assume $g$ is not eventually bounded by $\Phi_+ (A)$. We need to show $\Phi_- (g) \not\sub A$. To this end
we recursively construct an increasing sequence $(\tau_\alpha : \alpha < \kappa)$ such that $\tau_\alpha \in C_g^\alpha$
and $[\tau_{\alpha + 1} ] \cap A_\alpha = \emptyset$ for all $\alpha < \kappa$. Letting $x = \bigcup \tau_\alpha$,
we see $x \in \Phi_- (g) \sem A$.

Let $\tau_0 = \la \ra$.

Assume $\tau_\alpha$ has been defined as required. Since $g$ is not eventually bounded by $f_{\tau_\alpha}^A$, there is
$\gamma < \lambda$ such that $f_{\tau_\alpha}^A (\gamma) < g (\gamma)$. Thus we can find $\tau_{\alpha + 1}
\in C_g^{\alpha + 1}$ such that $h (\tau_\alpha \ha \sigma_\gamma) \subsetneq \tau_{\alpha + 1}$. 
$[\tau_{\alpha + 1} ] \cap A_\alpha = \emptyset$ follows because $[ h (\tau_\alpha \ha \sigma_\gamma) ] \cap 
A_{\lh ( \tau_\alpha \ha \sigma_\gamma)} = \emptyset$, $[\tau_{\alpha + 1} ] \sub [ h (\tau_\alpha \ha \sigma_\gamma) ] $, and
$\lh ( \tau_\alpha \ha \sigma_\gamma) > \lh (\tau_\alpha) \geq \alpha$
and thus $A_\alpha \sub A_{\lh ( \tau_\alpha \ha \sigma_\gamma)}$.

If $\alpha$ is a limit ordinal, simply let $\tau_\alpha = \bigcup_{\beta < \alpha} \tau_\beta$.
\end{proof}

\begin{cor}  \label{firstcor}
\begin{enumerate}
\item $\add (\M_\kappa) \leq \bb_\kappa$ and $\cof (\M_\kappa) \geq \dd^\lambda_\kappa$.
\item In particular, $\cof (\M_\kappa) \geq \dd_\kappa$.
\end{enumerate}
\end{cor}

\begin{proof}
(2) follows from (1) and $\dd^\lambda_\kappa \geq \dd_\kappa$. Also, if $2^{< \kappa} > \kappa$,
then $\add (\M_\kappa) = \kappa^+ \leq \bb_\kappa$, so the first inequality of (1) holds trivially.

To see the first inequality of (1) in case $2^{< \kappa} = \kappa$, let $\F \sub \kappa^\kappa$ be an unbounded family. If $A \in \M_\kappa$, there 
is $g \in \F$ not eventually bounded by $\Phi_+ (A)$ because $| \Phi_+ (A) | = \kappa$ and thus $\Phi_+ (A)$ is bounded.
So $\Phi_- (g) \not\sub A$. Thus we see that the union of the $\Phi_- (g)$, $g \in \F$, does not belong to $\M_\kappa$,
and $\add (\M_\kappa) \leq \bb_\kappa$ follows.

For the second inequality of (1), let $\A \sub \M_\kappa$, and assume $| \A | < \dd_\kappa^\lambda$. Let $\F = \bigcup \{ \Phi_+ (A) : A \in \A \}$.
Since $|\Phi_+ (A) | \leq \lambda$ for all $A$ and $\dd_\kappa^\lambda > \lambda$,  we also have $| \F | < \dd_\kappa^\lambda$. 
Hence there is $g \in \kappa^\lambda$, which is not eventually bounded by $\F$. But then $\Phi_- (g) \not\sub A$
for all $A \in \A$, and $\A$ is not cofinal in $\M_\kappa$. Thus $\cof (\M_\kappa) \geq \dd^\lambda_\kappa$ follows.
\end{proof}

\begin{cor}   \label{secondcor}
\begin{enumerate}
\item $\add (\M_\kappa) = \min \{ \bb_\kappa , \cov (\M_\kappa) \}$ and $\cof (\M_\kappa) \geq \max \{ \dd_\kappa^\lambda , \non (\M_\kappa) \}$\footnote{An earlier
version of this paper claimed equality here, but the proof was flawed.}.
\item If $2^{< \kappa} = \kappa$, then $\cof (\M_\kappa) = \max \{ \dd_\kappa , \non (\M_\kappa) \}$.
\end{enumerate}
\end{cor}

\begin{proof}
Since $\add (\M_\kappa) \leq \cov (\M_\kappa)$ is trivial and $\add (\M_\kappa ) \leq \bb_\kappa$ holds
by Corollary~\ref{firstcor} (1), the first part of (1) follows from~\cite[Corollary 31]{BBFM18} (see also item (ii) in the Introduction).
Similarly, the second part of (2) is immediate by Corollary~\ref{firstcor} (1) and the obvious $\cof (\M_\kappa) \geq \non (\M_\kappa)$.

For (2), use (1) and ~\cite[Corollary 31]{BBFM18} (see also item (iii) in the Introduction).
\end{proof}

Assume additionally $\lambda$ is regular. With an argument similar to the proof of Theorem~\ref{mainthm} we obtain:

\begin{prop}
There are functions $\Phi_- :  \lambda^\lambda \to \NWD_\kappa$ and $\Phi_+ : \M_\kappa \to 
\lambda^{\lambda}$ such that if $A \in \M_\kappa$, $g \in  \lambda^\lambda$, and $g$ is not eventually bounded by
$\Phi_+ (A)$, then $\Phi_- (g) \not\sub A$.
\end{prop}

\begin{proof}
Let $\Sigma$ be as in the proof of Theorem~\ref{mainthm}. Also let $\Tau = \{ \tau_\delta : \delta < \lambda\} = 2^{< \kappa}$.
Assume this enumeration satisfies additionally
\begin{itemize}
\item $\sigma \ha 0 = \tau_\delta$ implies $\sigma \ha 1 = \tau_{\delta + 1}$,
\item $\tau_\delta \subset \tau_\epsilon$ implies $\delta < \epsilon$.
\end{itemize}
Given $g \in \lambda^\lambda$, we define $C^\alpha_g \sub 2^{<\kappa}$, $T_g \sub 2^{< \kappa}$, and $\Phi_- (g) = [T_g]$
as in this proof, except for the successor step where for given $\sigma \in C^\alpha_g$, we first let $\tau \in D^{\alpha + 1}_g$ if for some
$\gamma < \lambda$, $\tau = \sigma \ha \sigma_\gamma \ha \tau_\delta$ for $\delta < g (\gamma)$ and then
define $C^{\alpha +1}_g$ as the set of all $\tau$ such that 
\begin{itemize}
\item either $\tau \in D^{\alpha +1}_g$ and no proper extension of $\tau$ belongs to $D^{\alpha + 1}_g$ 
\item or $\zeta: = \lh (\tau)$ is a limit ordinal, for all $\xi < \zeta$ with $\xi \geq \lh (\sigma)$, $\tau \re \xi \in D^{\alpha + 1}_g$ 
   and $\tau \notin D^{\alpha + 1}_g$.
\end{itemize}
It is then easy to see that $C^{\alpha + 1}_g$ is an antichain with the required properties. Note in particular that
for all $\tau \in D^{\alpha + 1}_g$ there is $\tau ' \supseteq \tau$ with $\tau ' \in C^{\alpha + 1}_g$.

(For suppose this fails for some $\tau$. Assume $\tau ' \supseteq \tau$ with $\tau ' \in  D^{\alpha + 1}_g$. Then
$\tau '$ is not a maximal node in $D^{\alpha + 1}_g$ and therefore $\tau '\ha 0 \in D^{\alpha + 1}_g$. Since the latter is
not maximal either, also $\tau'\ha 1 \in D^{\alpha + 1}_g$. Similarly if $\tau ' \supseteq \tau$ and $\lh (\tau ')$ is a limit such that
$\tau ' \re \xi \in  D^{\alpha + 1}_g$ for all $\xi$ with $\lh (\tau) \leq \xi < \lh (\tau ')$ then $\tau' \in D^{\alpha + 1}_g$,
for otherwise it would belong $C^{\alpha + 1}_g$. This means that the full binary tree below $\tau$ belongs to $D^{\alpha + 1}_g$,
a contradiction.)

Given $A = \bigcup_{\alpha < \kappa}  A_\alpha \in \M_\kappa$ as in the proof of Theorem~\ref{mainthm}, $\alpha < \kappa$, and $\sigma \in 2^{\alpha}$, define 
$f_\sigma^A \in \lambda^\lambda$ such that $[\sigma \ha \sigma_\gamma \ha \tau_{f_\sigma^A (\gamma)} ] \cap A_\alpha = \emptyset$. Let
$\Phi_+ (A)$ eventually dominate all $f_\sigma^A$. 

If $g$ is not eventually bounded by $\Phi_+ (A)$, we construct an increasing sequence $( \tau_\alpha : \alpha < \kappa)$ in $2^{< \kappa}$ with $\tau_\alpha \in C^\alpha_g$
and $[\tau_{\alpha + 1}] \cap A_\alpha = \emptyset$ as in the proof of Theorem~\ref{mainthm}. 
\end{proof}

As a consequence we immediately get (though this will follow from Corollary~\ref{firstcor} if Question~\ref{dom-order-question} has a positive answer):

\begin{cor} \label{thirdcor}
$\cof (\M_\kappa) \geq \dd_\lambda$.
\end{cor}

A family  $\F \sub \lambda^\lambda$ of functions is called {\em almost disjoint} if given any distinct $f , g \in \F$, the set $\{ \gamma < \lambda : f(\gamma ) = g (\gamma) \}$
has size less than $\lambda$.
Yet another modification of the proof of Theorem~\ref{mainthm} gives us an absolute result under certain assumptions.

\begin{thm} 
Assume $\lambda = | 2^{< \kappa} | \geq \kappa^{++}$. Let $\F \sub \lambda^\lambda$ be an almost disjoint family of functions. Then there is $\A \sub \NWD_\kappa$
of size $| \F |$ such that for all $B \in \M_\kappa$, the set $\{ A \in \A : A \sub B \}$ has size at most $\kappa$. In particular, if there is an almost disjoint family of functions
of size $2^\lambda$, then $\cof (\M_\kappa) = 2^\lambda$.
\end{thm}

\begin{proof}
Let $\Sigma = \{ \sigma_\gamma : \gamma < \lambda \} \sub 2^{< \kappa}$ be a maximal antichain in $2^{< \kappa}$.
By $\kappa < 2^{< \kappa}$, it is easy to see that we may additionally assume that $\Sigma$ forms a {\em front}, that is, 
for each $x \in 2^\kappa$ there is a (necessarily unique) $\sigma \in \Sigma$ with $\sigma \sub x$\footnote{This may be false 
if $\kappa = 2^{< \kappa}$. For example, if $\kappa$ is weakly compact, it is easy to see there is no front of size $\kappa$. This was pointed out
to us by Tristan van der Vlugt.}. 

Recursively define fronts $C^\alpha$, $\alpha < \kappa$, such that
\begin{romanenumerate}
\item $C^0 = \{ \la \ra \}$,
\item $C^{\alpha + 1} = \{ \sigma \ha \sigma_\gamma : \sigma \in C^\alpha$ and $\gamma  < \lambda \}$, and
\item $C^\alpha = \{ \sigma \in 2^{ < \kappa }  : \exists (\tau_\beta : \beta < \alpha)$ strictly increasing such that $\sigma = \bigcup_{\beta < \alpha } \tau_\beta$ and
   $\tau_\beta \in C^\beta \}$ for limit ordinals $\alpha$.
\end{romanenumerate}
Construe $f \in \F$ as a function from $2^{< \kappa}$ to $\lambda$. Define a nowhere dense tree $T_f \sub 2^{< \kappa}$ by recursion on $\alpha < \kappa$
by producing sets $C^\alpha_f \sub C^\alpha$ such that $C^0_f = C^0 = \{ \la \ra \}$, $C^\alpha_f$ is obtained from the previous $C^\beta_f$ as in (iii) for limit
ordinals $\alpha$, and
\begin{romanenumerate'} \setcounter{enuroman}{1}
\item $C^{\alpha + 1}_f  = \{ \sigma \ha \sigma_\gamma : \sigma \in C^\alpha_f$ and $f (\sigma ) \neq \gamma \}$.
\end{romanenumerate'}
Let $T_f$ be the downward closure of $\bigcup_{\alpha < \kappa} C^\alpha_f$ and put $A_f = [ T_f ]$. Since for $\sigma \in C^\alpha_f$, $\sigma \ha \sigma _{f (\sigma)} $
is incompatible with all elements of $C^{\alpha + 1 }_f$, we see that $[\sigma \ha \sigma_{f (\sigma)} ] \cap [T_f] = \emptyset$, and $A_f$ is indeed nowhere dense. 
Put $\A = \{ A_f : f \in \F \}$. 

Assume $B = \bigcup_{\alpha < \kappa} B_\alpha$ is a $\kappa$-meager set, where the $B_\alpha$ form an increasing sequence of nowhere dense sets.
We claim that $ \{ f \in \F : A_f \sub B \}$ has size at most $\kappa$.

Indeed, assume that $\{ f_\zeta : \zeta < \kappa^+ \} \sub \F$. We need to find $\zeta < \kappa^+$ and $x \in A_{f_\zeta} \sem B$.
By almost disjointness and $\lambda > \kappa^+$, we can find $\sigma \in \bigcap_{\zeta < \kappa^+} C^1_{f_\zeta}$ such that 
for all $\alpha \geq 1$, all $\sigma' \in C^\alpha$ with $\sigma \sub \sigma '$ and all distinct $\zeta , \zeta ' < \kappa^+$,
we have $f_\zeta (\sigma ') \neq f_{\zeta '} (\sigma ')$. Recursively build $\tau_\alpha \in C^\alpha$, $\alpha < \kappa$, such that
\begin{enumerate}
\item $\tau_0 = \la \ra$, $\tau_1 = \sigma$,
\item $\tau_\alpha = \bigcup_{\beta < \alpha} \tau_\beta$ for limit ordinals $\alpha$, and
\item if $\tau_\alpha$ has been constructed for $\alpha \geq 1$, find $\tau^* \supseteq \tau_\alpha$ such that $[ \tau^*] \cap B_\alpha = \emptyset$, let
   $\beta > \alpha$ and $\tau_\beta \supseteq \tau^*$ be such that $\tau_\beta \in C^\beta$, and define $\tau_\gamma \in C^\gamma$ for $\alpha < \gamma < \beta$ such that
   $\tau_\alpha \sub \tau_\gamma \sub \tau_\beta$.
\end{enumerate}
Put $x = \bigcup_{\alpha < \kappa} \tau_\alpha$. Clearly $x \notin B$.

Define a function $f$ on all $\tau_\alpha$, $\alpha < \kappa$, as follows.
\begin{enumerate}
\item $f (\tau_0) = f ( \la \ra )$ is the unique $\gamma$ such that $\tau_1 = \sigma = \sigma_\gamma$ and
\item $f (\tau_\alpha)$ is the unique $\gamma$ such that $\tau_\alpha \ha \sigma_\gamma = \tau_{\alpha + 1}$ for $\alpha \geq 1$.
\end{enumerate}
Now notice that by disjointness of the functions $f_\zeta$, there is $\zeta < \kappa^+$ such that for all $\alpha \geq 1$, we have $f_\zeta (\tau_\alpha) \neq f (\tau_\alpha)$.
This means, however, that for all $\alpha$, $\tau_\alpha$ belongs to $C^\alpha_{f_\zeta}$: indeed $\tau_1 = \sigma \in C^1_{f_\zeta}$, 
$\tau_\alpha \in C^\alpha_{f_\zeta}$ implies $\tau_{\alpha + 1} = \tau_\alpha \ha \sigma_{f (\tau_\alpha)} \in C^{\alpha + 1}_{f_\zeta}$ because $f_\zeta (\tau_\alpha) \neq f (\tau_\alpha)$,
and $\tau_\alpha \in C^\alpha_{f_\zeta}$ for limit $\alpha$ follows trivially. Therefore $x \in [T_{f_\zeta} ] = A_{f_\zeta}$, as required, and the proof is complete.
\end{proof}

By an old result of Baumgartner, almost disjoint families of functions of size $2^\lambda$ may not exist.
Indeed, if $\kappa < \lambda$ are regular and GCH holds, and we first add at least $\lambda^{++}$ Cohen subsets of $\kappa$ and then add $\lambda$ Cohen reals, then
$\lambda = | 2^{< \kappa}|$, $2^\lambda \geq \lambda^{++}$ and there is no almost disjoint family of functions in $\lambda^\lambda$ of size 
$\lambda^{++}$, by~\cite[Theorem 5.4 (a)]{Ba76}.


\section{Models}
\label{models}

{\sc Models for $\non (\M_\kappa)$.} We know~\cite[Observation 23 (iii)]{BBFM18} that $2^{<\kappa} \leq \non (\M_\kappa) \leq 2^\kappa$. We shall now
see that this is all we can say, even if $2^{< \kappa} > \kappa$.

In the model obtained by adding $\kappa^+$ Cohen reals over a model of GCH, we have $\kappa < 2^{< \kappa} = \non (\M_\kappa) = 2^\kappa$.

For a model with $\kappa < 2^{< \kappa} < \non (\M_\kappa) = 2^\kappa$, simply add $\kappa^{++} $ many $\kappa$-Hechler functions (see~\cite[Subsection 4.2]{BBFM18})
followed by $\kappa^+$ many Cohen reals to a model of GCH. In the intermediate model, $\bb_\kappa = \kappa^{++}$. Since Cohen forcing is $\kappa^\kappa$-bounding, 
it does not change the value of $\bb_\kappa$. Also $\bb_\kappa \leq \non (\M_\kappa)$ in ZFC. Therefore the final model satisfies $\non (\M_\kappa )
= 2^\kappa = \kappa^{++}$.

\begin{thm}  \label{1stcon}
It is consistent that $\kappa < 2^{< \kappa} = \non (\M_\kappa) < 2^\kappa$.
\end{thm}

\begin{proof}
Assume GCH in the ground model $V$. Add first $\kappa^{++}$ many Cohen subsets of $\kappa$ to obtain the model $V[G]$. Denote the forcing by $\CC_\kappa^{\kappa^{++}}$ 
and, more  generally, for $A \sub \kappa^{++}$, use $\CC_\kappa^A$ for the forcing adding the Cohen sets with index in $A$. Next add $\kappa^+$ many
Cohen reals to obtain the model $V[G][H]$. Work in the model $V[H]$. The forcing $(\CC_\kappa^{\kappa^{++}})^V$ is still $< \kappa$-distributive in
this model (though not $< \kappa$-closed anymore) and $\kappa^+$-cc; in particular, it does not add new sequences of length $< \kappa$. 
In $V[H]$, $2^{< \kappa} = 2^\kappa = \kappa^+$, and we claim that $2^\kappa \cap V[H]$ is a witness for $\non (\M_\kappa)$ in
$V[G][H]$. 

Let $( \dot A_\alpha : \alpha < \kappa)$ be a $(\CC_\kappa^{\kappa^{++}})^V$-name for an increasing sequence of nowhere dense sets. Thus there is
a name $\dot f$ for a function from $2^{< \kappa}$ to $2^{< \kappa}$ such that for $\alpha < \kappa$ and $\sigma \in 2^\alpha$, the trivial
condition forces $\sigma \sub \dot f (\sigma)$ and $[ \dot f (\sigma) ] \cap \dot A_\alpha = \emptyset$ Without loss, $\lh (\dot f (\sigma))$
is forced to be at least $\alpha +1$. Let $p \in (\CC_\kappa^{\kappa^{++}})^V$.
Recursively produce sets $C_\alpha \sub \kappa^{++}$ with $C_\alpha \in V$, conditions $p_\alpha , q_\alpha \in (\CC_\kappa^{\kappa^{++}})^V$, and
sequences $\tau_\alpha \in 2^{< \kappa}$, $\alpha < \kappa$, such that 
\begin{itemize}
\item $|C_\alpha| = \kappa$, and the $C_\alpha$ are an increasing chain,
\item $p_\alpha \leq p, q_\alpha \leq p_\alpha$,
\item $p_\alpha \in (\CC_\kappa^{C_\alpha})^V , q_\alpha \in (\CC_\kappa^{C_{\alpha+1}})^V$, and for every $q \in (\CC_\kappa^{C_\alpha})^V$ there is a $\beta \geq \alpha$
such that $q = p_\beta$,
\item $\tau_\alpha \in 2^{\geq\alpha}$ and $\tau_\alpha \sub \tau_\beta$ for $\alpha \leq \beta$,
\item $q_\alpha \forces [\tau_{\alpha + 1}] \cap \dot A_\alpha = \emptyset$.
\end{itemize}
Note that for $C \sub \kappa^{++}$ of size $\kappa$, $| (\CC_\kappa^{C})^V| = (2^{<\kappa})^V = \kappa$, and thus the second clause
of the third item can easily be achieved by a book-keeping argument.

Let $C_0 \sub \kappa^{++}$ be such that $C_0 \in V$ and $p \in (\CC_\kappa^{C_0})^V$. Put $\tau_0 = \la\ra$. Let $p_0 \leq p$ be the condition in $(\CC_\kappa^{C_0})^V$ handed
down by the book-keeping, and find $q_0 \leq p_0$ and $\tau_1$ such that $q_0$ forces $\tau_1 = \dot f (\tau_0)$ (note that this is possible because 
no new $<\kappa$-sequences are added). In particular, $q_0 \forces [\tau_1] \cap \dot A_0 =\emptyset$.
Let $C_1 \supseteq C_0$ be such that $C_1 \in V$ and $q_0 \in (\CC_\kappa^{C_1})^V$.

Assume we are at stage $\alpha$, and everything has been constructed for $\beta < \alpha$. In case $\alpha$ is successor, we also
assume $C_\alpha$ and $\tau_\alpha$ have been produced, and if $\alpha $ is limit we let $C_\alpha \supseteq \bigcup_{\beta < \alpha} C_\beta$
with $C_\alpha \in V$ and $\tau_\alpha = \bigcup_{\beta < \alpha} \tau_\beta$. Let $p_\alpha \leq p$ be the condition in $(\CC_\kappa^{C_\alpha})^V$ 
given by the book-keeping, and proceed as in the basic step to get $q_\alpha \leq p_\alpha$, $\tau_{\alpha + 1} \supseteq \tau_\alpha$,
and $C_{\alpha +1} \supseteq C_\alpha$ such that $C_{\alpha + 1}  \in V$ and $q_\alpha \in (\CC_\kappa^{C_{\alpha+1}})^V$ forces $\tau_{\alpha + 1} = \dot f (\tau_\alpha)$ and thus
$[\tau_{\alpha + 1} ] \cap \dot A_\alpha =\emptyset$. This completes the recursive construction.

Now let $C = \bigcup_{\alpha < \kappa} C_\alpha$ and $x = \bigcup_{\alpha < \kappa} \tau_\alpha \in 2^\kappa$. We claim that
$p$ forces that $x \notin \dot A$ where $\dot A = \bigcup_{\alpha < \kappa} A_\alpha$. Let $\beta < \kappa$. Take any $q' \leq p$ in $(\CC_\kappa^{\kappa^{++}})^V$
and let $q = q' \re C$. Thus $q = p_\alpha$ for some $\alpha \geq \beta$ (by the book-keeping). By construction, $q_\alpha \leq p_\alpha$
forces that $[\tau_\alpha ] \cap \dot A_\alpha = \emptyset$. Clearly $q_\alpha$ and $q'$ are compatible. Let $q'' = q_\alpha \cup q'$ be 
the smallest common extension. Then $q''$ forces $x \notin \dot A_\beta$. Since this holds for any $q' \leq p$ and any $\beta$,
$p$ forces $x \notin \dot A$.

Hence $2^\kappa \cap V[H] $ is indeed non-meager in $V[G][H]$.
\end{proof}

The proof of the following result is somewhat more complicated.

\begin{thm}   \label{2ndcon}
It is consistent that $\kappa < 2^{< \kappa} < \non (\M_\kappa) < 2^\kappa$.
\end{thm}

\begin{proof}
Again assume GCH. First add $\kappa^{+++}$ Cohen subsets of $\kappa$. Then perform a $\kappa^{++}$-stage iteration of $\kappa$-Hechler
forcing (see~\cite[Subsection 4.2]{BBFM18}). In the resulting model $V[G_0][G_1]$, $2^{< \kappa} = \kappa < \bb_\kappa = \non (\M_\kappa) =
\kappa^{++} < 2^\kappa = \kappa^{+++}$. Next add $\kappa^+$ Cohen reals to make $2^{< \kappa} = \kappa^+$ in the final model
$V[G_0][G_1][H]$. Again it is clear this model will satisfy $\bb_\kappa = \kappa^{++}$, so that $\non (\M_\kappa) \geq \kappa^{++}$,
and it suffices to show $\non (\M_\kappa) \leq \kappa^{++}$. In $V[H]$ the remainder forcing is $< \kappa$-distributive and
$\kappa^+$-cc. 

It is well-known (see e.g.~\cite[Subsection 4.2]{BBFM18}) that $\kappa$-Hechler forcing also adjoins a $\kappa$-Cohen function.
In $V[G_0][G_1]$ let $( c_\alpha \in \kappa^\kappa : \alpha < \kappa^{++} )$ be the Cohen functions decoded from the $\kappa^{++}$
many Hechler functions added in the iteration. We will use the $c_\alpha$ to code $c_\alpha^\gamma \in 2^\kappa$, $\kappa\leq\gamma < \kappa^+$,
in such a way that the set $C = \{ c_\alpha^\gamma : \alpha < \kappa^{++} , \kappa \leq \gamma < \kappa^+ \}$ is non-meager in $V[G_0][G_1][H]$.

More explicitly, let $f : \kappa^+ \to 2^{< \kappa}$ be a bijection with $f \in V[H]$. By distributivity of the remainder forcing,
$f$ is still a bijection in $V[G_0][G_1][H]$. Assume $f(0) = \la\ra$. Next, for each $\gamma$ with $\kappa \leq \gamma < \kappa^+$, let 
$g_\gamma : \kappa \to \gamma$ be a bijection with $g_\gamma \in V$ and $g_\gamma(0) = 0$. Define $c_\alpha^\gamma \in V[G_0][G_1][H]$
to be the concatenation of the $f ( g_\gamma (c_\alpha (\zeta)))$ where $\zeta < \kappa$, i.e.
\[ c_\alpha^\gamma = f ( g_\gamma (c_\alpha (0))) \ha f ( g_\gamma (c_\alpha (1))) \ha f ( g_\gamma (c_\alpha (2))) \ha ... \ha 
f ( g_\gamma (c_\alpha (\zeta))) \ha ... \]
Note that if $c_\alpha (\zeta) \neq 0$ then $f(g_\gamma(c_\alpha (\zeta))) \neq \la\ra$ so that by genericity $\kappa$ many of the
$f ( g_\gamma (c_\alpha (\zeta)))$ are non-empty sequences and $c_\alpha^\gamma$ is indeed an element of $2^\kappa$.

To see that $C$ is non-meager, let $A$ be a meager set in $V[G_0][G_1][H]$, say $A = \bigcup_{\zeta < \kappa} A_\zeta$ where 
the $A_\zeta$ form an increasing sequence of nowhere dense sets.  Thus there is $h : 2^{< \kappa} \to 2^{< \kappa}$ such that
for all $\zeta < \kappa$ and all $\sigma \in 2^\zeta$, $[\sigma \ha h (\sigma) ] \cap A_\zeta = \emptyset$. By the $\kappa^{+}$-cc
of the remainder forcing, there is $\alpha < \kappa^{++}$ such that $h \in V[H] [G_0] [G_1^\alpha]$, where $G_1^\alpha$ is the generic
for the $\alpha$ first stages of the iteration of $\kappa$-Hechler forcing. That is, $c_\alpha$ is still $\CC_\kappa^V$-generic over 
$V_\alpha : = V[H] [G_0] [G_1^\alpha]$. Work in $V_\alpha$. We shall show that for some $\gamma < \kappa^+$, $\dot c_\alpha^\gamma$
is forced to be outside $A$.

Let $\gamma < \kappa^+$ with $\gamma \geq \kappa$. For $\tau \in \gamma^{< \kappa}$ let $\bar f (\tau)$ be the concatenation
of the $f (\tau (\zeta))$, $\zeta < \lh (\tau)$, i.e.
\[ \bar f (\tau) = f (\tau (0)) \ha f (\tau (1)) \ha f (\tau (2)) \ha ... \ha f (\tau (\zeta)) \ha ... \]
Say that $\gamma$ is {\em $h$-good} if
\begin{itemize}
\item $\delta < \gamma$ implies that $f^{-1} (h ( f (\delta))) < \gamma$,
\item $\tau \in ( \gamma^{< \kappa} \cap V )$ implies that $f^{-1} (\bar f (\tau)) < \gamma$.
\end{itemize}

\begin{claim}
There are $h$-good $\gamma$.
\end{claim}

\begin{proof}
This is a standard closure argument. Let $\gamma_0$ be arbitrary with $\kappa \leq \gamma_0 < \kappa^+$.
Recursively construct an increasing continuous sequence $(\gamma_\zeta : \zeta \leq \kappa)$ of ordinals
below $\kappa^+$ such that 
\begin{itemize}
\item if $\delta < \gamma_\zeta$ then $f^{-1} (h ( f (\delta))) < \gamma_{\zeta + 1}$,
\item if $\tau \in ( \gamma_\zeta^{< \kappa} \cap V )$ then $f^{-1} (\bar f (\tau)) < \gamma_{\zeta + 1}$.
\end{itemize}
Since $ | \gamma_\zeta^{< \kappa} \cap V | = \kappa$, this is possible. Clearly $\gamma_\kappa$ is $h$-good.
\end{proof}

\begin{claim}
If $\gamma$ is $h$-good, then $\dot c_\alpha^\gamma$ is forced to be outside $A$.
\end{claim}

\begin{proof}
Let $\zeta < \kappa$ and let $\upsilon \in \CC_\kappa^V = \kappa^{<\kappa} \cap V$ be a $\kappa$-Cohen condition.
We need to find $\upsilon ' \leq \upsilon $ such that $\upsilon ' \forces \dot c_\alpha^\gamma \notin A_\zeta$.
Let $\tau \in \gamma^{< \kappa} \cap V$ be the image of $\upsilon$ under $g_\gamma$, that is, 
$\lh (\tau) = \lh (\upsilon)$ and $\tau (\zeta) = g_\gamma (\upsilon (\zeta))$ for all $\zeta < \lh (\upsilon)$.
Let $\sigma : = \bar f (\tau)$. Without loss of generality we may assume that $\eta : = \lh (\sigma) \geq \zeta$;
otherwise extend the condition $\upsilon$. Clearly $\upsilon \forces  \dot c_\alpha^\gamma \in [ \sigma ]$.
Since $\gamma$ is $h$-good, we know that
$\delta : = f^{-1} (\sigma) < \gamma$. Note $f (\delta) = \sigma$. We also have that $\epsilon : = f^{-1} (h (\sigma)) < \gamma$.
Note $f(\epsilon) = h (\sigma)$. By definition of $h$, we obtain $[\sigma \ha h(\sigma) ] \cap A_\eta = \emptyset$
and therefore also $[\sigma \ha h(\sigma) ] \cap A_\zeta = \emptyset$. 
Let $\tau ' = \tau \ha \epsilon$ and $\upsilon' = \upsilon \ha g_\gamma^{-1} (\epsilon)$. 
Thus $\bar f (\tau') = \bar f (\tau) \ha f (\epsilon) = \sigma \ha h(\sigma)$ and $\upsilon ' \forces  \dot c_\alpha^\gamma \in [ \sigma\ha h (\sigma) ]$.
In particular $\upsilon ' \forces \dot c_\alpha^\gamma \notin A_\zeta$.
\end{proof}

This completes the proof of the theorem.
\end{proof}

\bigskip

\noindent {\sc Models for $\cof (\M_\kappa)$.} As before let $\lambda = |2^{< \kappa}|$. By the results of Section~\ref{ZFC} we know in
particular that $\lambda < \cof (\M_\kappa) \leq 2^\lambda$. We are interested in models with $\kappa < \lambda$. 

If we add 
$\kappa^+$ Cohen reals over a model of GCH, we obtain a model of $\kappa < \lambda = 2^\kappa = \kappa^+ < 
\cof(\M_\kappa) = 2^\lambda = \kappa^{++}$.

For the consistency of $\kappa < \lambda < \cof (\M_\kappa) = 2^\kappa$, use the model of Theorem~\ref{1stcon}. 

To obtain a model of $\kappa <  \lambda  < 2^\kappa < \cof (\M_\kappa)$, first add $\kappa^{+++}$ Cohen subsets of $\kappa^+$, then
$\kappa^{++}$ Cohen subsets of $\kappa$, and finally $\kappa^+$ Cohen reals over a model for GCH. For $\lambda = \kappa^+$, $\dd_\lambda = \kappa^{+++}$
in the first extension, and this is preserved. Therefore the final model satisfies $\dd_\lambda = \cof (\M_\kappa) = 2^\lambda = \kappa^{+++}$,
by Corollary~\ref{thirdcor}.

\begin{ques}  \label{cof-ques}
Is $\kappa < 2^{<\kappa}$ together with $\cof (\M_\kappa) < 2^{2^{< \kappa}}$ consistent?
\end{ques}

In view of the results Section~\ref{ZFC}, this is related to questions about the dominating numbers in the next section.


\section{Dominating numbers}
\label{dominating}

For this section, let $\kappa$ be regular (not necessarily uncountable) and $\lambda \geq \kappa$ arbitrary. 
Let us first see that in the definition of $\dd^\lambda_\kappa$ it does not matter whether we use everywhere domination or domination modulo
$< \kappa$ or $< \lambda$ (in case $\cf(\lambda)\geq\kappa$). 
For $f,g \in \kappa^\lambda$, say $f \leq_{\lambda } g$ if there is $\delta < \lambda$ such that for all $\gamma \geq \delta$,
$f(\gamma) \leq g(\gamma)$. 
Let $\dd^\lambda_\kappa (\leq)$ be the dominating number of $\kappa^\lambda$ with the everywhere domination ordering,
and let $\dd_\kappa^\lambda (\leq_\lambda)$ be the dominating number of $\kappa^\lambda$ with the ordering $\leq_\lambda$.

\begin{prop}   \label{dom-prop}
$\dd_\kappa^\lambda (\leq) = \dd_\kappa^\lambda (\leq_\lambda)$. In particular, if $\cf (\lambda) \geq \kappa$, 
$\dd_\kappa^\lambda = \dd_\kappa^\lambda (\leq) = \dd_\kappa^\lambda (\leq_\lambda)$.
\end{prop}

\begin{proof}
The second statement follows from the first because $\cf (\lambda) \geq \kappa$ obviously implies  $\dd_\kappa^\lambda (\leq_\lambda)
\leq \dd_\kappa^\lambda \leq \dd_\kappa^\lambda (\leq)$. So it suffices to see that $\dd_\kappa^\lambda (\leq) \leq \dd_\kappa^\lambda (\leq_\lambda)$.

Take $\F \sub \kappa^\lambda$ dominating in $(\kappa^\lambda, \leq_\lambda)$. For $f,g \in \F$ and $\alpha,\beta < \lambda$ define the
function $h_{f,g,\alpha,\beta} \in \kappa^\lambda$ by
\[ h_{f,g,\alpha,\beta} (\gamma) = \left\{ \begin{array}{ll} g(\alpha + \gamma) & \mbox{if } \gamma < \beta \\
f(\gamma) & \mbox{if } \gamma \geq \beta. \\ \end{array} \right. \]
Since $| \{ h_{f,g,\alpha,\beta} : f,g \in \F , \alpha,\beta < \lambda \} | = | \F|$, it suffices to show that this family is dominating
everywhere. To this end let $h \in \kappa^\lambda$ be arbitrary. There are $f \in \F$ and $\beta < \lambda$ such that
$f(\gamma) \geq h(\gamma)$ for all $\gamma \geq \beta$. Now partition $\lambda$ into intervals $I_\zeta$, $\zeta < \lambda$, so
that each $I_\zeta$ has length exactly $\beta$. Let $i_\zeta = \min I_\zeta$ for all $\zeta$. Define $h' \in \kappa^\lambda$ by
\[ h' (i_\zeta + \xi ) = h (\xi) \]
for all $\zeta < \lambda$ and all $\xi < \beta$. There are $g \in \F$ and $\eta < \lambda$ such that $g(\delta) \geq h'(\delta)$ for all $\delta \geq i_\eta$.
Then, for $\xi < \beta$,
\[ h_{f,g,i_\eta,\beta} (\xi) = g(i_\eta + \xi) \geq h'(i_\eta + \xi) = h (\xi), \]
and we see that $h_{f,g,i_\eta,\beta}$ dominates $h$ everywhere.
\end{proof}

Since $\cf (2^{< \kappa}) \geq \kappa$ the second statement is true when $\lambda = |2^{< \kappa}|$ as in Section~\ref{ZFC}.

\begin{thm}   \label{dom-thm}
Let $\kappa$ be regular and $\lambda > \kappa$. Then $\dd^\lambda_\kappa \geq \dd^\lambda_{\kappa^+}$.
\end{thm}

\begin{proof}
We do the proof using Tukey connections. First define $\Phi_- (f) : \kappa^\lambda \to (\kappa^+)^\lambda$ by
\[ \Phi_- (f) (\alpha) = \min \{ \gamma < \kappa^+ : \exists \zeta < \kappa \; \exists^\kappa \beta \in [\alpha , \alpha + \gamma ) \mbox{ such that } f (\beta) = \zeta \} \]
for $f \in \kappa^\lambda$ and $\alpha < \lambda$. Here $\exists^\kappa \beta$ denotes the quantifier ``there are $\kappa$ many $\beta$". 
This is clearly well-defined by the pigeonhole principle.

For defining $\Phi_+ : (\kappa^+)^\lambda \to \kappa^\lambda$, fix $h \in ( \kappa^+ )^\lambda$. Assume without loss that $h (\alpha) > 0$
for all $\alpha$. Let $(\alpha_\xi : \xi < \lambda)$ be the strictly increasing club
sequence in $\lambda$ recursively defined by $\alpha_0 = 0$, $\alpha_{\xi + 1} = \alpha_\xi + h (\alpha_\xi)$, and $\alpha_\eta =
\bigcup_{\xi < \eta} \alpha_\xi$ for limit ordinals $\eta$. Notice that for $\eta < \lambda$ we always have $\alpha_\eta < \lambda$ (even for
singular $\lambda$) because the sequence increases only by ordinals of size at most $\kappa$. Replacing $h$ recursively by a larger
function, if necessary, we may also assume that for every $\xi < \lambda$ and every $\alpha \in [\alpha_\xi , \alpha_{\xi + 1})$,
$h (\alpha_{\xi + 1}) \geq h (\alpha)$. Now define $\Phi_+ (h)$ such that for every $\xi < \lambda$, $\Phi_+ (h) \re [\alpha_\xi , \alpha_{\xi + 1})$
is a one-to-one function into $\kappa$. 

We first claim that if $ h \not\leq^* \Phi_- (f)$ then $\Phi_+ (h) \not\leq^* f$.

To see this let $\alpha < \lambda$ be such that $\Phi_- (f) (\alpha) < h (\alpha)$. There are at least $\kappa^+$ many such $\alpha$.
Let $\xi < \lambda$ be such that $\alpha \in [\alpha_\xi , \alpha_{\xi + 1})$. By choice of $h$ we have $\Phi_- (f) (\alpha) < h(\alpha) 
\leq h(\alpha_{\xi + 1} )$ so that also \[\alpha + \Phi_- (f) (\alpha) < \alpha + h(\alpha)  \leq \alpha_{\xi + 1} + h(\alpha_{\xi + 1} ) =
\alpha_{\xi + 2}. \]
In particular there is $\zeta < \kappa$ such that $f$ assumes value $\zeta$ exactly $\kappa$ many times in the interval $[\alpha_\xi , \alpha_{\xi + 2} )$.
Since $\Phi_+ (h)$ is one-to-one on both intervals $[\alpha_\xi , \alpha_{\xi + 1} )$ and $[\alpha_{\xi+1} , \alpha_{\xi + 2} )$, it follows that there
are $\kappa$ many places where $\Phi_+ (h)$ is above $f$. Thus $\Phi_+ (h) \not\leq^* f$.

As a consequence, it now readily follows that if $\F \sub \kappa^\lambda$ is dominating then so is 
$\{ \Phi_- (f) : f \in \F \}$ in $(\kappa^+)^\lambda$.
\end{proof}

\begin{ques}   \label{dom-order-question}
Let $\kappa$ and $\mu$ be regular and $\kappa < \mu \leq \lambda$. Does $\dd^\lambda_\kappa \geq \dd^\lambda_\mu$ hold?
In particular, is $\dd^\lambda_\kappa \geq \dd_\lambda$ for regular $\lambda$?
\end{ques}

The inequality in Theorem~\ref{dom-thm} is consistently strict. 

\begin{obs}
Assume GCH.
Let $\kappa_0 < \kappa_1 < ... < \kappa_n$ be regular cardinals, let $\lambda \geq \kappa_n$, and let $\mu_n < ... < \mu_1 < \mu_0$
be cardinals with $\mu_n > \lambda$ and $\cf (\mu_i) > \lambda$.
Then there is a forcing extension with $\dd_{\kappa_i}^\lambda = \mu_i$.
\end{obs}

\begin{proof}
Start by adding $\mu_n$ many Cohen subsets of $\kappa_n$. By backwards recursion add $\mu_i$ many Cohen subsets of $\kappa_i$.
Finish by adding $\mu_0$ many Cohen subsets of $\kappa_0$. The first forcing forces $2^{\kappa_n} = 2^\lambda = \dd_{\kappa_n}^\lambda = \mu_n$. 
Since the remainder forcing is $\kappa_{n-1}^+$-cc and thus $\kappa_n^{\kappa_n}$-bounding, it preserves the value of $\dd^\lambda_{\kappa_n}$.
Iterating this argument we see that $2^{\kappa_i} = 2^\lambda =  \mu_0$ and $\dd_{\kappa_i}^\lambda = \mu_i$ in the final model.
\end{proof}

In particular $\dd^\lambda_\kappa < 2^\lambda$ is consistent for $\kappa < \lambda$, but in the model provided by the observation $2^{< \kappa} > \dd^\lambda_\kappa
> \lambda$ holds. On the other hand, for an affirmative answer to Question~\ref{cof-ques} we would need a model with $\lambda := 2^{<\kappa} > \kappa$
and $\dd^\lambda_\kappa < 2^\lambda$. Let us formulate this question somewhat more generally:

\begin{ques}
Assume $\kappa$ is regular and $\lambda > \kappa$ with $\lambda \geq 2^{<\kappa}$. Is $\dd^\lambda_\kappa < 2^\lambda$ consistent?
\end{ques}

For $\kappa = \omega$ and $\lambda = \omega_1$ this is a famous old question of Jech and Prikry~\cite{JP79} (see also~\cite[Problem 8.1]{Mi15}):

\begin{ques}[Jech, Prikry]
Is $\dd^{\omega_1}_\omega < 2^{\omega_1}$ consistent?
\end{ques}

The dominating numbers $\dd^\lambda_\kappa$ have since been used in~\cite{BSta}.



\end{document}